\documentclass[oneside,10pt]{amsart}

\usepackage[utf8]{inputenc}
\usepackage[scale=0.7]{geometry}
\usepackage[colorlinks=true,linkcolor=blue]{hyperref}
\usepackage{microtype}
\usepackage{graphicx}
\usepackage{amsmath}
\usepackage{color}
\usepackage{multirow}

\newcommand{\abs}[1]{\lvert #1 \rvert}

\newcommand{\Z}{\mathbb Z}
\DeclareMathOperator{\SR}{SR}
\DeclareMathOperator{\sr}{br}
\DeclareMathOperator{\Sidon}{F_2}
\DeclareMathOperator{\Bh}{F_h}
\DeclareMathOperator{\BR}{BR}

\makeatletter
\def\hlinewd#1{%
\noalign{\ifnum0=`}\fi\hrule \@height #1 \futurelet
\reserved@a\@xhline}
\makeatother

\newtheorem{theorem}{Theorem}
\newtheorem{lemma}[theorem]{Lemma}

\title{Sidon-Ramsey and $B_{h}$-Ramsey numbers}

\author[M. A. Espinosa-García]{Manuel A. Espinosa-García}
\address[M. A. Espinosa-García]{Centro de Ciencias Matemáticas, UNAM Campus Morelia, Morelia, Mexico}
\email{esgama@matmor.unam.mx}

\author{Amanda Montejano}
\address[A. Montejano]{Facultad de Ciencias, UNAM campus Juriquilla, Querétaro, Mexico}
\email{amandamontejano@ciencias.unam.mx}

\author[E. Roldán-Pensado]{Edgardo Roldán-Pensado}
\address[E. Roldán-Pensado]{Centro de Ciencias Matemáticas, UNAM Campus Morelia, Morelia, Mexico}
\email{e.roldan@im.unam.mx}

\author[J. D. Suárez]{J. David Suárez}
\address[J. D. Suárez]{Facultad de Ciencias, UNAM campus Juriquilla, Querétaro, Mexico}
\email{suar\_david@hotmail.com}

\thanks{This research was supported by CONACyT project 282280 and PAPIIT project IG100822.}

\begin{document}

\begin{abstract}
	For a given positive integer $k$, the Sidon-Ramsey number $\SR(k)$ is defined as the minimum value of $n$ such that, in every partition of the set $[1, n]$ into $k$ parts, there exists a part that contains two distinct pairs of numbers with the same sum. In other words, there is a part that is not  a Sidon set. In this paper, we investigate the asymptotic behavior of this parameter and two generalizations of it. The first generalization involves replacing pairs of numbers with $h$-tuples, such that in every partition of $[1, n]$ into $k$ parts, there exists a part that contains two distinct $h$-tuples with the same sum. Alternatively, there is a part that is not a $B_h$ set. The second generalization considers the scenario where the interval $[1, n]$ is substituted with a non-necessarily symmetric $d$-dimensional box of the form $\prod_{i=1}^d[1,n_i]$.  For the general case of $h\geq 3$ and non-symmetric boxes, before applying our method to obtain the Ramsey-type result, we needed to establish an upper bound for the corresponding density parameter.
\end{abstract}

\maketitle

\section{Introduction}

A subset $S$ of an additive group $G$ is called a \emph{Sidon set} if the sums of any two elements (possibly equal) of $S$ are distinct. 
In other words, if $x,y,z,w\in S$ satisfy
\[x+y=z+w,\]
then $\{x,y\}=\{z,w\}$, which means that the equation above has only trivial solutions in $S$.
For a given $X\subset G$, an important problem is to determine the maximum size of a Sidon set contained in $X$. This problem has been mainly studied when $G=\Z$ and $X=[1,n]$. We use $\Sidon(n)$ to denote the size of the largest Sidon set contained in $[1,n]$. It is known that
\begin{equation}\label{eq:bestS}
	n^{1/2}(1-o(1))\leq \Sidon(n)\leq n^{1/2}+0.998\,n^{1/4}.
\end{equation}
The upper bound has been progressively improved \cite{ET41,Lin69,Cil10}, the best being recently established by Balogh, Füredi and Roy \cite{BFR23}. The lower bound may be inferred from several known constructions; in particular the one provided by Singer concerning maximal Sidon sets in $X=G=\Z_n$, where $n=q^2+q+1$ and $q$ is a prime power \cite{Sin38}. For more information about problems related with Sidon sets the reader may consult the survey paper of O'Bryant \cite{OBr04}.

As with many density theorems, there is a Ramsey version of the problem of maximizing the size of a Sidon set. For a given positive integer $k$, a \emph{Sidon $k$-partition} of $X\subset G$ is a partition of $X$ into $k$ parts, all of which are Sidon sets. Let $\SR(k)$ be the \emph{Sidon-Ramsey number}, defined as the minimum $n$ such that there is no Sidon $k$-partition of $[1,n]$. This parameter can be found in different contexts under different names. For instance, the existence of the Sidon-Ramsey numbers is a consequence of a theorem of Rado \cite{Rad43} with the matrix
\[\begin{pmatrix}
	1&1&-1&-1&0&0\\
	1&0&-1&0&1&0\\
	1&0&0&-1&0&1
\end{pmatrix},\]
where two new variables and two new equations are introduced to assure that no trivial solutions to the equation $x+y=z+w$ are considered. Liang et al. \cite{LLXX13} and Xu et al. \cite{XLL18}, using computer assistance, found the exact values of $\SR(k)$ for $k\le 5$ and gave specific bounds for $k\le 19$. They use the fact that $\SR(k)$ can be bounded from above using the pigeonhole principle, that is, 

\begin{equation}\label{eq:chinos}
	\SR(k)\leq (t-1)k+1 \text{ for any } k,t \ge 2 \text{ satisfying } (\Sidon(t)-1)/(t-1)>k.
\end{equation}

By computing the upper bound derived from \eqref{eq:chinos} and \eqref{eq:bestS}, along with the well-known Singer construction of Sidon sets, it can be easily deduced that the Sidon-Ramsey number $\SR(k)$ behave asymptotically as $k^{2}$ (Theorem \ref{thm:main}).  Due to the interesting connection between Sidon sets and $C_4$-free graphs, the proof of the lower bound of Theorem \ref{thm:main}  closely resembles the proof presented in \cite{CG75} for Theorem 3. However,  we have included this proof in Section \ref{sec:2} to ensure completeness.

\begin{theorem}\label{thm:main} 
	Let $k$ be a positive integer, then
	\[k^2 - O(k^{c})\le \SR(k)\le k^2 + C k^{3/2} + O(k),\]
	where $c\le 1.525$ depends on the distribution of the prime numbers and $C$ can be taken close to $1.996$ and depends on the best upper bound for Sidon numbers.
\end{theorem}

In this paper, we investigate analogous results  to Theorem \ref{thm:main} corresponding to some generalizations of $F_2(n)$. Specifically, we study the Ramsey-type parameter for $B_h$ sets in the interval $X=[1,n]\subset\Z$ (precise definitions will be provided in Section \ref{sec:2}). Additionally, we explore the problem in higher dimensions by studying the size of the largest $B_{h}$-set within a non-necessarily symmetric $d$-dimensional box  $X=\prod_{j=1}^{d}[1,n_{j}]\subset\Z^d$.

The paper is organized as follows.  In Section \ref{sec:2}  we deal with everything related to the one dimensional case. First, we define a $B_h$ set in $\Z$, its density parameter $\Bh(n)$, and its corresponding Ramsey-type parameter $\BR_h(k)$. If $h=2$ this coincides with the definition of a Sidon set, $\Sidon(n)$, and  $\SR(k)$. Then, we present the proof of Theorem \ref{thm:main} which states the asymptotic behavior of $\SR(k)$ (corresponding to the case $h=2$) separated from the proof of Theorem \ref{thm:bh},  which states the asymptotic behavior of $\BR_h(k)$ for  $h\geq 3$.

Section \ref{sec:3} is  devoted to the work in higher dimensions.  Given positive integers $n_{1}\le \dots \le n_{d}$, we use $\Bh(n_1,\dots, n_d)$ to denote the size of the largest $B_h$-set contained in a $d$-dimensional box $[1,n_{1}]\times \dots \times [1,n_{d}]$.  If  $n_1= \dots=n_d=n$,  we said that the box is symmetric. Previous research by Lindström \cite{Lin72}, Cilleruelo \cite{Cil10}, and Rackham and {\v{S}}arka \cite{RS10} has studied this parameter in specific scenarios: symmetric boxes with $h=2$, asymmetrical boxes with $h=2$, and symmetric boxes for $h\geq 2$, respectively. To complete this investigation, we employ techniques developed by Cilleruelo \cite{Cil10} and Rackham and {\v{S}}arka \cite{RS10} to establish an upper bound for $\Bh(n_1,\dots, n_d)$ in the most general case, where the box is not necessarily symmetric and $h\geq 3$ (refer to Theorem \ref{cajasfh}). For the Ramsey-type version, it is convenient  to define $\sr_h(n_1,\dots,n_d)$ to be the largest positive integer $k$ such that there is no $k$-partition of $\prod_{i\le d}[1,n_i]$ in which all its parts are $B_h$ sets. Note that, in dimension $d=1$, $\sr_h(n)$ is the counterpart of $\BR_h(k)$ in the sense that $\sr_h(\BR_h(k))=k$.

Tables 1 and 2 serve to identify the notation that we will use and locate the bounds of each parameter in the manuscript.

\begin{table}[ht]
\begin{center}
\begin{tabular}{|c||c|c|}
\hline
&$h=2$ & $h\geq 3$ \\
\hlinewd{1pt}
$\Bh(n)$ & \eqref{eq:bestS}&  \eqref{Fhcota}\\
\hline
$\BR_h(k)$ &  Theorem \ref{thm:main} &  Theorem \ref{thm:bh}\\
\hline
\end{tabular}
\end{center}
\caption{We use $\Bh(n)$ to denote the size of the largest $B_h$-set contained in $[1,n]$, and so $\Sidon(n)$ is the classical Sidon number.  The corresponding Ramsey-type parameter is denoted by $\BR_h(k)$, and so $\SR(k)=\BR_2(k)$.}
\label{table:d=1}
\end{table}

\begin{table}[ht]
\begin{tabular}{|c|c||c|c|}
\hline
& &$h=2$  & $h\ge 3$\\ 
\hlinewd{1pt}
$\Bh(n_1,\dots, n_d)$& $n_1= \dots=n_d=n$ &   \eqref{eq:Linboth}&  \eqref{eq:Bhbound}\\
\cline{2-4}
& $N=n_1n_2 \dots n_d$ & \eqref{eq:Sidonmot}  &   Theorem \ref{cajasfh}\\
\hlinewd{1pt}
$\sr_h(n_1,\dots, n_d)$& $n_1= \dots=n_d=n$ &    \eqref{Teo6-Sim-h=2} &  \eqref{teo6-sim} \\ 
\cline{2-4}
& $N=n_1n_2 \dots n_d$  &  \eqref{Teo6 final} &  Theorem \ref{cajasbrh}\\
\hline
\end{tabular}
\caption{We use $\Bh(n_1,\dots, n_d)$ to denote the size of the largest $B_h$-set contained in a $d$-dimensional box $[1,n_{1}]\times [1,n_{2}]\times \dots \times [1,n_{d}]$.  The corresponding Ramsey-type parameter is denoted by  $\sr_h(n_1,\dots, n_d)$. When $n_1=n_2= \dots=n_d=n$ we said that the box is symmetric. In other case, the results are expressed in terms of the product $N=n_1n_2 \dots n_d$.}
\label{table:d>2}
\end{table}

\section{\texorpdfstring{$B_h$}{Bh} and \texorpdfstring{$B_h$}{Bh}-Ramsey numbers in intervals}\label{sec:2}

A subset $S$ of an additive group $G$ is called a \emph{$B_h$-set} if all sums of the form $a_1+\dots+a_h$, where $a_1,\dots, a_h\in A$, are distinct; note that a $B_2$-set is a Sidon set. We use $\Bh(n)$ to denote the size of the largest $B_h$-set contained in $[1,n]$. It is known that 
\begin{equation}\label{Fhcota}
	(1+o(1))n^{1/h}\le \Bh(n)\le c_hn^{1/h}.
\end{equation}
The lower bound was proved by Bose and Chowla \cite{BC62}, while the constant $c_h$ in the upper bound has been successively improved \cite{DR84,Jia93,Che94,Cil01}. Currently, the best bounds are due to Green \cite{Gre01}, who proved that $c_3<1.519$, $c_4<1.627$ and $c_h\le \frac 1{2e}(h+(\frac{3}{2}+o(1))ln(h))$.

As with Sidon sets, we can define the Ramsey version of this problem. For a given positive integer $k$ and a subset $X$ of an additive group $G$, a \emph{$B_h$ $k$-partition} of $X$ is a partition of $X$ into $k$ parts, all of which are $B_h$-sets. Let $\BR_h(k)$ be the \emph{$B_h$-Ramsey number}, defined as the minimum $n$ such that there is no $B_h$ $k$-partition of $[1,n]$.
We shall note that $\BR_2(k)=\SR(k)$.

\begin{proof}[Proof of Theorem \ref{thm:main}]
	From the upper bound in \eqref{eq:bestS} and using the pigeonhole principle we deduce that the interval $[1,n]$ cannot be partitioned into less than 
	\[k\ge\frac{n}{\Sidon(n)}\ge \frac{n}{n^{1/2}+0.998\,n^{1/4}}\]
	Sidon sets. This can be simplified, using the geometric series, to
	\begin{align*}
		k &\ge \frac{n^{1/2}}{1+0.998\,n^{-1/4}} \\
		&= n^{1/2}\left(1-0.998\,n^{-1/4}+O\left(n^{-1/2}\right)\right)\\
		&= \left(n^{1/4}-0.499\right)^2+O(1).
	\end{align*}
	Solving this for $n$ gives
	\begin{align*}
		n &\le \left((k+O(1))^{1/2}+0.499\right)^4\\
		&= (k+O(1))^2+1.996(k+O(1))^{3/2}+O(k)\\
		&= k^2+1.996\,k^{3/2}+O(k),
	\end{align*}
	which provides the upper bound.
	
	To prove the lower bound, notice that a Sidon set in $\mathbb{Z}_{n}$ is also a Sidon set in $[1,n]\subset\Z$. Therefore, any Sidon-Ramsey $k$-partition of $\mathbb{Z}_{n}$ induces a Sidon-Ramsey $k$-partition of $[1,n]$. Singer proved that, if $q$ is a prime power, there exists a Sidon set $S$ in $\mathbb{Z}_{q^{2}+q+1}$ of size $q+1$ \cite{Sin38}. Assume that $S=\{s_{1},s_{2},\dots,s_{q+1}\}$ and consider the Sidon sets $S_{i}=S-s_{i}$. Since $|S_i|=|S|=q+1$ and $S_i\cap S_j=\{0\}$, with $i\neq j$, we have that 
	\[\big\vert\bigcup_{i=1}^{q+1}S_{i}\big\vert=(q+1)(q+1)-(q+1)+1=q^2+q+1,\]
	then $\bigcup_{i=1}^{q+1}S_{i}=\Z_{q^2+q+1}$ and so $\bigcup_{i=1}^{q+1}S_{i}$ covers $\Z_{q^2+q+1}$. Therefore, we may construct a Sidon-Ramsey $(q+1)$-partition by taking subsets of the elements of this cover.
	Let $k\in\mathbb{N}$ and let $p$ be the largest prime less than $k$. It is known that $p=k-O(k^{0.525})$ (see e.g. \cite{BHP01}). As there exists a Sidon-Ramsey $(p+1)$-partition of $[1,p^{2}+p+1]$, we conclude that $\SR(k)\ge \SR(p+1)\ge p^2+p+1= k^2-O(k^{1.525})$, which completes the proof.
\end{proof}

For $h\ge 2$ we have a result analogous to the Theorem \ref{thm:main}. We can derive an upper bound of $\BR_h(k)$ from Green's bound \cite{Gre01} and the pigeonhole principle. The lower bound comes from considering translates of the $B_{h}$-set constructed by Ruzsa (for $h=2$) \cite{Ruz93} and by G\'omez-Trujillo (for $h>2$) \cite{GT11}. Note that the construction used for the $h=2$ case here, is not the same as the one used in the proof of Theorem \ref{thm:main}, although it gives the same bound.
\begin{theorem}\label{thm:bh}
	Let $k,h$ be positive integers, $h\ge 2$, then
	\[k^{\frac{h}{h-1}}-O\left(k^{1+\frac{c}{h-1}}\right)\le \BR_h(k)\le C_hk^{\frac{h}{h-1}},\]
	where $c\le 0.525$ depends on the distribution of the prime numbers and $C_h$ depends on the best upper bound for $\Bh$ numbers.
\end{theorem}
\begin{proof}
	
First, we work the upper bound for the $B_{h}-Ramsey$ numbers. The case $h=2$ follows from Theorem \ref{thm:main}. If $h>2$, the best known upper bound for $B_h$ sets, established by Green in 2001\cite{Gre01}, is
\[F_h(n)\leq c_hn^{1/h},\]
where $c_h\le \frac{1}{2e}\left(h+\left(3/2+o(1)\right)\log(h)\right)$.
From this and using the pigeonhole principle we deduce that the interval $[1,n]$ cannot be partitioned into less than 
\[k\ge\frac{n}{F_h(n)}\ge\frac{n}{c_hn^{1/h}}\]
$B_h$ sets. 
Solving this for $n$ gives
\[n \le C_hk^{\frac{h}{h-1}},\]
where $C_h=c_h^{\frac{h}{h\textcolor{red}{-}1}}$.
Which provides the upper bound in Theorem \ref{thm:bh} for $h>2$.

To prove the lower bound in Theorem \ref{thm:bh}, notice that a $B_{h}$ set in $\mathbb{Z}_{n}$ is also a $B_{h}$ set in $[1,n]\subset\mathbb{Z}$. We conclude that any $B_{h}$-Ramsey $k$-partition of $\mathbb{Z}_{n}$ induces a $B_{h}$-Ramsey $k$-partition of $[1,n]$. The main idea to get the lower bounds. In general, for any positive integers $a\leq b$ and any function $f:\mathbb{Z}_{a}\to \mathbb{Z}_{b}$, the set \[C=\{(t,f(t)):t\in \mathbb{Z}_{a}\}\] and their translations $C+(0,1)$, $C+(0,2)$, $\dots$, $C+(0,b-1)$ give a partition of $\mathbb{Z}_{a}\times \mathbb{Z}_{b}$. Ruzsa proved that for any $p$ prime number, there is a function $f:\mathbb{Z}_{p-1}\to \mathbb{Z}_{p}$, such that the set
\[\{(a,f(a)): a\in \mathbb{Z}_{p-1}\}\subseteq \mathbb{Z}_{p-1}\times \mathbb{Z}_{p}\] is a $B_{2}$-set (see \cite{Ruz93}). For $h>2$, G\'omez and Trujillo proved that for any $p$ prime number there exist a function $f:\mathbb{Z}_{p}\to \mathbb{Z}_{p^{h-1}-1}$ such that the set
\[\{(a,f(a)): a\in \mathbb{Z}_{p}\}\subseteq \mathbb{Z}_{p}\times \mathbb{Z}_{p^{h-1}-1}\] is a $B_{h}$-set (see \cite{GT11}). Then, for any prime number $p$, there is a $B_{2}$ Ramsey partition of $[1,p^{2}-p]$ in $p$ parts, and for $h>2$ there is a $B_{h}$-Ramsey partition of $[1,p^{h}-p]$ in $p^{h-1}-1$ parts.

Let $k\in \mathbb{N}$ and, for $h>2$, let $p$ be the largest prime number such that $p^{h-1}-1\le k$ (for $h=2$ we consider $p$ to be the largest prime number less than $k$). It is known that $p=(k+1)^{1/(h-1)}-O((k+1)^{0.525/(h-1)})$ (for $h=2$, we have $p=k-O(k^{0.525})$) (see \cite{Dus99}). In the case that $h>2$ we have that \begin{align*}BR_{h}(k)&\ge \BR_{h}(p^{h-1}-1)\ge p^{h}-p\\ &=((k+1)^{1/(h-1)}-O((k+1)^{0.525/(h-1)}))^h-(k+1)^{1/(h-1)}+O((k+1)^{0.525/(h-1)})\\
	&=(k+1)^{h/(h-1)}-O((k+1)^{1+0.525/(h-1)}).\end{align*}
In the case $h=2$ we have that
\[BR_{2}(k)\ge \BR_{2}(p)\ge p^{2}-p\ge (k-O(k^{0.525}))^{2}-(k-O(k^{0.525}))=k^{2}-O(k^{1.525}).\qedhere\]
\end{proof}

\section{\texorpdfstring{$B_h$}{Bh} and \texorpdfstring{$B_h$}{Bh}-Ramsey numbers in \texorpdfstring{$d$}{d}-dimensional boxes}\label{sec:3}

It is an interesting problem to study $B_h$-sets in higher dimensions. For a fixed positive integer $d$ and positive integers $n_{1}\le \dots \le n_{d}$, we seek to bound the largest cardinality $\Bh(n_{1},\dots,n_{d})$ of a $B_{h}$-set in the $d$-dimensional (not necessarily symmetric) box $$X=[1,n_{1}]\times [1,n_{2}]\times \dots \times [1,n_{d}]=\prod_{j=1}^{d}[1,n_{j}]\subset\Z^d.$$

Regarding the lower bounds, there is a natural way to map one-dimensional $B_h$-sets to $d$-dimensional  $B_h$-sets.
Any integer $0\le a\le n_{1}\cdots n_{d}-1$ has exactly one representation of the form $a_{1}+a_{2}n_{1}+a_{3}n_{1}n_{2}+\cdots +a_{d-1}n_{1}n_{2}\dots n_{d-1}$, where $0\le a_{i}< n_{i}$. Let $\varphi:\mathbb{Z}\to\mathbb{Z}^{d}$ be such that $\varphi(a)=(a_{1},a_{2},\dots ,a_{d})$. Note that any $B_h$-set in $[0,n_{1}\cdots n_{d}-1]$ gets sent into a $B_h$-set in $[0,n_{1}-1]\times \cdots \times [0,n_{d}-1]$, since $\varphi(x_{1})+\cdots +\varphi(x_{h})=\varphi(y_{1})+\cdots +\varphi(y_{h})$ implies that $x_{1}+\cdots +x_{h}=y_{1}+\cdots+y_{h}$.
Using this property of $\varphi$, we immediately have that
\begin{equation}\label{eq:lowFh}
	\Bh(n_1\cdots n_d)\le\Bh(n_1,\dots,n_d).
\end{equation}
So lower bounds in the $d$-dimensional case can be obtained from lower bounds in the one-dimensional case. This was observed for Sidon sets by Cilleruelo \cite{Cil10}.

As for the upper bounds, the first result was given by Lindström for the case $n_1=\dots=n_d=n$ and $h=2$ \cite{Lin72}, i.e. Sidon sets in symmetrical boxes in high dimensions. The bound obtained was
\begin{equation}\label{eq:Lin}
	\Sidon(n,\dots,n)\leq n^{d/2}+O\left(n^{\frac{d^{2}}{2d+2}}\right),
\end{equation}
which together with \eqref{eq:lowFh} and \eqref{Fhcota} gives
\begin{equation}\label{eq:Linboth}
n^{d/2}(1-o(1))\leq \Sidon(n,\dots,n)\leq n^{d/2}+O\left(n^{\frac{d^{2}}{2d+2}}\right).
\end{equation}

The best result obtained for asymmetrical boxes with $h=2$ is due to Cilleruelo \cite{Cil10}, he proved that
\begin{equation}\label{eq:Cil}
	\Sidon(n_1,\dots,n_d)\le N^{1/2}\left(1+O\left(\left(\frac{N_{s-1}}{N^{1/2}}\right)^{\frac{1}{d-s+2}}\right)\right),
\end{equation}
where $N_{0}=1$, $N_{i}=\prod_{j=1}^{i}n_{j}$ for $1\le i\le d$, $N=N_{d}$ and $s$ is the least integer such that $N^{1/2}\le n_{s}^{d-s+2}N_{s-1}$. This match the  lower bound obtained by \eqref{eq:lowFh} and \eqref{Fhcota}, 
\begin{equation}\label{eq:Sidonmot}
N^{1/2}(1-o(1))\leq 	\Sidon(n_1,\dots,n_d)\le N^{1/2}\left(1+O\left(\left(\frac{N_{s-1}}{N^{1/2}}\right)^{\frac{1}{d-s+2}}\right)\right).
\end{equation}

For $h\ge 2$, the best upper bound so far for symmetrical boxes is due to Rackham and {\v{S}}arka \cite{RS10} who showed that
\begin{equation}\label{eq:RS}
	\Bh(n,\dots,n)\le
	\begin{cases}
		t^{\frac{d}{h}}(t!)^{\frac{1}{t}}n^{\frac{d}{h}}+O\left(n^\frac{d^2}{h(d+1)}\right) & \text{if $h=2t$,}\\
		t^{\frac{d-1}{h}}(t!)^{\frac{2}{h}}n^{\frac{d}{h}}+O\left(n^{\frac{d^2}{h(d+1)}}\right) & \text{if $h=2t-1$,}\\
	\end{cases}
\end{equation}
which again match the  lower bound obtained from \eqref{eq:lowFh} and \eqref{Fhcota},
\begin{equation}\label{eq:Bhbound}
\left(1+o(1)\right)n^{d/h}\le	\Bh(n,\dots,n)\le
	\begin{cases}
		t^{\frac{d}{h}}(t!)^{\frac{1}{t}}n^{\frac{d}{h}}+O\left(n^\frac{d^2}{h(d+1)}\right) & \text{if $h=2t$,}\\
		t^{\frac{d-1}{h}}(t!)^{\frac{2}{h}}n^{\frac{d}{h}}+O\left(n^{\frac{d^2}{h(d+1)}}\right) & \text{if $h=2t-1$.}\\
	\end{cases}
\end{equation}
For large enough $h$,  Rackham and {\v{S}}arka  gave the improvement
\begin{equation}\label{eq:RSlarge}
	\Bh(n,\dots,n)\le
	\begin{cases}
		(\pi d)^{\frac{d}{2h}}(1+\epsilon(h))t^{\frac{d}{2h}}(t!)^{\frac{2}{h}}n^{\frac{d}{h}}+O\left(n^\frac{d^2}{h(d+1)}\right) & \text{if $h=2t$,}\\
		(\pi d)^{\frac{d}{2h}}(1+\epsilon(h))t^{\frac{d-2}{2h}}(t!)^{\frac{2}{h}}n^{\frac{d}{h}}+O\left(n^\frac{d^2}{h(d+1)}\right) & \text{if $h=2t-1$,}\\
	\end{cases}
\end{equation}
where $\epsilon(h)$ is a function that approach $0$ as $h\to \infty$.

Note that the bound \eqref{eq:Cil} by Cilleruelo in the symmetrical case and the bound \eqref{eq:RS} by Rackham and \v{S}arka for $h=2$ coincide with the bound \eqref{eq:Lin} by Lindström. In Theorem \ref{cajasfh} we provide an upper bound for $\Bh(n_1,\dots, n_d)$ in the most general case, where the box is not necessarily symmetric and $h\geq 2$. We use a mix of techniques used by Cilleruelo and Rackham and \v{S}arka. Before continuing  we need a couple of definitions and lemmas.


Given subsets $A$ and $B$ of an additive  group $G$, and $g \in G$, we define the \emph{sumset} of $A$ and $B$ as
\begin{equation*}
A+B=\{a+b:a\in A, b\in B\}, 
\end{equation*}
and the \emph{additive energy} between $A$ and $B$ as
\begin{equation*}
	\sum_{g\in G}d_{A}(g)d_{B}(g),
\end{equation*}
where \begin{equation*}
	d_{X}(g)=\lvert\{(x,x'):x,x'\in X, x-x'=g\}\rvert.
\end{equation*}

\begin{lemma}[\cite{Cil10}]\label{energiaaditiva}
	Let $G$ be an additive group and let $A,B\subset G$. Then
	\[\lvert A\rvert^{2}\le \frac{\lvert A+B\rvert}{\lvert B\rvert^{2}}\sum_{g\in G}d_{A}(g)d_{B}(g).\]
\end{lemma}

The following lemma is a generalization of Lemmas 4.2.1 and 4.3.1 from \cite{RS10}. The two lemmas from \cite{RS10} require that $B$ is a symmetrical box, but this fact is not actually used in the proof. We use $tA$  to denote the set of sums of $t$  (not necessary different) elements of $A$, and  $t\ast A$  to denote the set of sums of $t$ distinct elements of $A$.

\begin{lemma}\label{cajaB} 
	Let $h\ge 2$, $A$ be a $B_{h}$ set in $\mathbb{N}^{d}$ and $B=[0,i_{1}-1]\times [0,i_{2}-1]\times \cdots \times [0,i_{d}-1]$.
	\begin{enumerate}
		\item If $h=2t$ then 
		\[\sum_{z\in\Z^{d}}d_{tA}(z)d_{B}(z)\le \lvert B\rvert ^{2}+O(\lvert B\rvert \lvert A\rvert ^{h-1}).\]
		\item If $h=2t-1$ then 
		\[\sum_{z\in\Z^{d}}d_{t\ast A}(z)d_{B}(z)\le \frac{\lvert A\rvert }{t}\lvert B|^{2}+O(\lvert B|\lvert A|^{h}).\]
	\end{enumerate}
\end{lemma}

Using Lemma \ref{energiaaditiva} and Lemma \ref{cajaB}  together with techniques previously used by Cilleruelo \cite{Cil10} and Rackham and {\v{S}}arka \cite{RS10} we obtain the following theorem.

\begin{theorem}\label{cajasfh}
	Let $n_{1}\le n_2\le\dots\le n_{d}$ be positive integers, set $N_{0}=1$, $N_{i}=\prod_{j=1}^{i}n_{j}$ for $1\le i\le d$, $N=N_{d}$, and let $s$ be the least integer such that $N^{1/h}\le n_{s}^{d-s+2}N_{s-1}$. Then, for $h\ge 2$,
	\[
	\Bh(n_{1},\dots,n_{d})\le
	\begin{cases}
		(t!)^{\frac{2}{h}}t^{\frac{d}{h}}N^{\frac{1}{h}}\left(1+O\left(\left(\frac{N_{s-1}}{N^{\frac{1}{h}}}\right)^{\frac{1}{d-s+2}}\right)\right) & \text{if $h=2t$,}\\
		(t!)^{\frac{2}{h}}t^{\frac{d-1}{h}}N^{\frac{1}{h}}\left(1+O\left(\left(\frac{N_{s-1}}{N^{\frac{1}{h}}}\right)^{\frac{1}{d-s+2}}\right)\right) & \text{if $h=2t-1$.}\\
	\end{cases}
	\]
\end{theorem}

\begin{proof}
	First, we will work the case $h=2t$. Let $A$ be a $B_{h}$ set in $X=\prod_{j=1}^{d}[1,n_{j}]\subset\Z^d$, $1\le s\le d$ and let $r_{j}=0$ for $j<s$ and $r_{j}=\lfloor tn_{j}M\rfloor$ for $j\ge s$ and some $0<M<1$ fixed. Applying Lemma \ref{energiaaditiva} with $tA$ and $B=[0,r_{1}]\times \cdots \times [0,r_{d}]$, and using Lemma \ref{cajaB} we have that
	\begin{equation}\label{eq:thm5par}\frac{\abs{tA}^{2}\abs{B}^{2}}{\abs{tA+B}}\le \sum_{z\in\Z^{d}}d_{tA}(z)d_{B}(z)\le \abs{B}^{2}+O(\abs{B}\abs{A}^{2t-1}).\end{equation}
	The number of elements in $tA$ is $\binom{\abs{A}-1+t}{t}$. A basic lower estimation to this binomial coefficient is $\frac{\abs{A}^{t}}{t!}\le \binom{\abs{A}-1+t}{t}$, then it follows that $\frac{\abs{A}^{t}}{t!}\le \abs{tA}$. Similarly, we get that $\frac{\abs{A}^{h}}{h!}\le \abs{hA}$ and $hA\subset hX$, which implies $\frac{\abs{A}^{h}}{h!}\le \abs{hX}= h^{d}N$, i.e. $\abs{A}=O(N^\frac{1}{h})$. Then we can estimate \eqref{eq:thm5par} as
	\[\frac{\abs{A}^{2t}\abs{B}^{2}}{(t!)^{2}\abs{tA+B}}\le \abs{B}^{2}+O(\abs{B}\abs{A}^{2t-1})\le \abs{B}^{2}+O(\abs{B}\abs{N}^{\frac{h-1}{h}})\]
	or equivalently
	\[\abs{A}^{2t}\le (t!)^{2}\abs{tA+B}\left(1+O\left(\frac{N^{\frac{h-1}{h}}}{\abs{B}}\right)\right).\]
	Notice that $tA+B\subset [1,tn_{1}+r_{1}]\times\cdots \times[1,tn_{d}+r_{d}]$. As $r_{j}\le tn_{j}M< r_{j}+1$ then
	\begin{align*}
		\abs{tA+B}&\le \prod_{j=1}^{d}(tn_{j}+r_{j})=t^{d}N\prod_{j=1}^{d}\left(1+\frac{r_{j}}{tn_{j}}\right)= t^{d}N\prod_{j=s}^{d}\left(1+\frac{r_{j}}{tn_{j}}\right)\\
		&\le t^{d}N(1+M)^{d-s+1}=t^{d}N(1+O(M)),
	\end{align*}
	and
	\[\abs{B}=\prod_{j=s}^{d}(1+r_{j})\ge \prod_{j=s}^{d}tn_{j}M=t^{d-s+1}\frac{N}{N_{s-1}}M^{d-s+1}.\]
	We conclude that
	\[\abs{A}^{2t}\le (t!)^{2}t^{d}N(1+O(M))\left(1+O\left(\frac{N^{\frac{h-1}{h}}N_{s-1}}{NM^{d-s+1}}\right)\right).\]
	To minimize the order of the last expression, take $M=(N_{s-1}/N^{1/h})^{1/(d-s+2)}$ and $s$ as the least integer such that $N^{1/h}\le 
	n_{s}^{d-s+2}N_{s-1}$. Then
	\[\abs{A}^{2t}\le 
	(t!)^{2}t^{d}N\left(1+O\left(\left(\frac{N_{s-1}}{N^{\frac{1}{h}}}\right)^{\frac{1}{d-s+2}}\right)\right).\]
	We conclude the proof of this case by taking the $2t^{\text{th}}$ root.	
	
	The proof in the case $h=2t-1$ is pretty similar to the proof in the case $h=2t$. Let $A$ be a $B_{h}$-set in $X$, $1\le s\le d$ and $0<M<1$. We take $M$, $r_{j}$ and $B$ as the even case proof. Applying Lemma \ref{energiaaditiva} with $t\ast A$ and $B$ and using Lemma \ref{cajaB} we have that 
	\[\frac{\abs{t\ast A}^{2}\abs{B}^{2}}{\abs{t\ast A+B}}\le \sum_{z\in\Z^{d}}d_{t\ast A}(z)d_{B}(z)\le \frac{\abs{A}}{t}\abs{B}^{2}+O(\abs{B}\abs{A}^{2t-1}).\]
	The number of elements in $t\ast A$ is $\binom{A}{t}$. A basic lower estimation of the binomial coefficient is
	$\frac{\abs{A}^{t}}{t!}(1-\frac{c_{t}}{\abs{A}})\le \abs{t\ast A}$ where $c_{t}$ depends only on $t$. Using that, and the facts that $\abs{B}=O(N)$ and $\abs{A}=\Omega(N^{\frac{1}{h}})$ (this last follows from \eqref{Fhcota} and \eqref{eq:lowFh}, and from $\frac{\abs{A}^{h}}{h!}\le h^{d}N$), we have that
	\begin{align*}
		\frac{\abs{A}^{2t}\left(1-\frac{c_{t}}{\abs{A}}\right)^{2}\abs{B}^{2}}{(t!)^{2}\abs{t\ast A+B}}\le \frac{\abs{A}}{t}\left(\abs{B}^{2}+O(\abs{B}\abs{A}^{2t-2})\right),
	\end{align*}
	or equivalently
	\begin{align*}
		\abs{A}^{2t}\le& \frac{(t!)^{2}\abs{A}}{t}\abs{t\ast A+B}\left(1+O\left(\frac{\abs{A}^{2t-2}}{\abs{B}}\right)\right)\left(\frac{1}{1-\frac{c_{t}}{\abs{A}}}\right)\\
		\le& \frac{(t!)^{2}\abs{A}}{t}\abs{t\ast A+B}\left(1+O\left(\frac{\abs{A}^{2t-2}}{\abs{B}}\right)\right)\left(1+O\left(\frac{1}{\abs{A}}\right)\right)\\
		\le& \frac{(t!)^{2}\abs{A}}{t}\abs{t\ast A+B}\left(1+O\left(\frac{N^{\frac{h-1}{h}}}{\abs{B}}\right)\right)
	\end{align*}

	Similarly as in the even case, we get that $\abs{B}\ge t^{d-s+1}\frac{N}{N_{s-1}}M^{d-s+1}$ and $\abs{t\ast A+B}\le t^{d}N{(1+O(M))}$. Then
	\[\abs{A}^{2t-1}\le \frac{(t!)^{2}}{t}t^{d}N(1+O(M))\left(1+O\left(\frac{N^{\frac{2t-2}{2t-1}}N_{s-1}}{NM^{d-s+1}}\right)\right).\]
	
	In order to minimize the last expression we take $M=(N_{s-1}/N^{1/h})^{1/(d-s+2)}$ and $s$ as the least integer such that $N^{1/h}\le n_{s}^{d-s+2}N_{s-1}$. Then 
	\[\abs{A}^{2t-1}\le (t!)^{2}t^{d-1}N\left(1+O\left(\left(\frac{N_{s-1}}{N^{\frac{1}{h}}}\right)^{\frac{1}{d-s+2}}\right)\right).\]
	We conclude the proof taking the $(2t-1)^{\text{th}}$ root.
\end{proof}

When $h$ is large enough, it is sometimes possible to improve Theorem \ref{cajasfh}. If the box $X=\prod_{j=1}^{d}[1,n_{j}]$ is sufficiently symmetric, we may consider a slightly larger symmetrical box $X'$ that contains $X$, then it is possible to use \eqref{eq:RSlarge} instead of \eqref{eq:RS} to obtain the improvement.

For the Ramsey version it is convenient to bound the size of the partition in terms of the dimensions of the box.
Let $\sr_h(n_1,\dots,n_d)$ be the largest positive integer $k$ such that there is no $B_h$ $k$-partition of $\prod_{i\le d}[1,n_i]$.
In dimension $d=1$, $\sr_h(n)$ is the counterpart of $\BR_h(k)$ in the sense that $\sr_h(\BR_h(k))=k$.

The second result of this section is a lower and upper bound for $\sr_h(n_1,\dots, n_d)$. 
\begin{theorem}\label{cajasbrh}
	Given positive integers $n_{1}\le n_2\le\dots\le n_{d}$, let $N_0=1$, $N_i=\prod_{j\le i}n_j$ for $1\leq i\leq d$, $N=N_d$, and let $s$ be the least index such that $N^{1/h}\le n_s^{d-s+2}N_{s-1}$. Define
	\begin{equation*}
		\gamma_h(n_{1},\dots,n_{d}) = \begin{cases} \frac{N^{\frac{h-1}{h}}}{\left(t!\right)^{\frac{2}{h}}t^{\frac{d}{h}}}\left(1-O\left(\left(\frac{N_{s-1}}{N^{\frac{1}{h}}}\right)^{\frac{1}{d-s+2}}\right)\right)&\text{if }h=2t,\\
			\frac{N^{\frac{h-1}{h}}}{\left(t!\right)^{\frac{2}{h}}t^{\frac{d-1}{h}}}\left(1-O\left(\left(\frac{N_{s-1}}{N^{\frac{1}{h}}}\right)^{\frac{1}{d-s+2}}\right)\right)&\text{if }h=2t-1,\end{cases}
	\end{equation*}
	then
	\[\gamma_h(n_{1},\dots,n_{d}) \le \sr_h(n_{1},\dots,n_{d}) \le N^{\frac{h-1}{h}}\left(1+O\left(N\right)^{\frac{-1+c}{h}}\right),\]
	where $c\le 0.525$ depends on the distribution of the prime numbers.
\end{theorem}

\begin{proof}
	The lower bound is obtained by using the pigeonhole principle with the bounds given in Theorem \ref{cajasfh}. The upper bound follows from the lower bound in Theorem \ref{thm:bh} and the fact that any $B_{h}$ $k$-partition in $[1,N]$ can be mapped to a $B_{h}$ $k$-partition in $[1,n_{1}]\times \cdots\times [1,n_{d}]$. 
\end{proof}

Therefore, the $B_h$-Ramsey number $\sr_h(n_{1},\dots,n_{d})$ behaves asymptotically as $N^{(h-1)/h}$.

Note that the lower bound in Theorem \ref{cajasbrh} depends on the best upper bound for $F_h$. In particular, for large enough $h$, it is sometimes possible to use \eqref{eq:RSlarge} instead of \eqref{eq:RS} to obtain a better lower bound for $\sr_h$. 
To finish the paper we evaluate some particular cases of Theorem \ref{cajasbrh}. When $h=2$ and the box is symmetric,
\begin{equation}\label{Teo6-Sim-h=2}
n^{d/2}-O\left(n^{\frac{d^2}{2d+2}}\right)\le\sr_2(n,\dots,n)\le n^{d/2}+O\left(n^{dc/2}\right).
\end{equation}
When $h=2$ and the box is not necessarily symmetric,
\begin{equation}\label{Teo6 final}
N^{1/2}\left(1-O\left(\frac{N_{s-1}}{N^{1/2}}\right)^{\frac{1}{d-s+2}} \right)\le\sr_2(n_1,\dots,n_d)\le N^{1/2}+O\left(N^{c/2}\right).
\end{equation}
 Finally, when $h\ge 3$ and the box is symmetric
	\begin{equation*}
		\gamma_h(n,\dots,n) = \begin{cases} \frac{n^{d\frac{(h-1)}{h}}-O\left(n^{\frac{d(h-1)}{h}-\frac{d}{h(d+1)}}\right)}{\left(t!\right)^{\frac{2}{h}}t^{\frac{d}{h}}}&\text{if }h=2t,\\
			\frac{n^{d\frac{(h-1)}{h}}-O\left(n^{\frac{d(h-1)}{h}-\frac{d}{h(d+1)}}\right)}{\left(t!\right)^{\frac{2}{h}}t^{\frac{d-1}{h}}}&\text{if }h=2t-1,\end{cases}
	\end{equation*}
and then,
\begin{equation}\label{teo6-sim}
\gamma_h(n,\dots,n)\le \sr_h(n,\dots,n)\le n^{\frac{d(h-1)}{h}}+O\left(n^{\frac{d(c+h-2)}{h}}\right).
\end{equation}


\bibliographystyle{amsalpha}

\bibliography{elbueno3107}

\end{document}